\documentclass[12pt,reqno]{amsart}

\usepackage[hmargin=2.6cm,bmargin=2.55cm,tmargin=3.05cm]{geometry}
\usepackage[usenames]{color}
\usepackage{graphbox}
\usepackage{subcaption}

\usepackage{mathtools}
\numberwithin{equation}{section}

\usepackage{amsmath,amssymb,amsthm,amsfonts,enumerate,url,mathbbol,mathrsfs,esint,tikz,graphicx,pifont,float}
\usetikzlibrary{calc}
\usepackage{hyperref}
\usepackage{multirow}

\usepackage{float}

\usepackage[normalem]{ulem}
\usepackage[utf8]{inputenc}

\theoremstyle{plain}
\newtheorem{theorem}{Theorem}[section]
\newtheorem{lemma}[theorem]{Lemma}

\theoremstyle{definition}
\newtheorem{definition}[theorem]{Definition}
\newtheorem{example}[theorem]{Example}

\newtheorem{remark}[theorem]{Remark}

\numberwithin{equation}{section}
\numberwithin{figure}{section}
\numberwithin{table}{section}

\newcommand{\R}{\mathbb{R}}

\newcommand{\Hp}{\mathcal H_p}
\newcommand{\Hpa}{\mathcal{H}_{p, \alpha}}
\newcommand{\Hpaa}{\mathcal{H}_{p, \alpha_0}}
\newcommand{\hHpa}{\widehat{\mathcal{H}}_{p, \alpha}}
\newcommand{\hHpaa}{\widehat{\mathcal H}_{p, \alpha_0}}
\newcommand{\HT}{T^* \, \nabla_f^2 L \, T}
\newcommand{\HS}{S^* \, \frac{1}{p}\nabla_{\lambda, f}^2 L \, S}
\newcommand{\phip}{\phi_p}
\newcommand{\DG}{\mathsf{G}} 
\newcommand{\DE}{\mathsf{E}} 
\newcommand{\DV}{\mathsf{V}} 
\newcommand{\CE}{\mathsf{E}_{\text{cut}}} 
\newcommand{\CG}{\widehat{\mathsf{G}}} 

\DeclareMathOperator{\Span}{span}
\DeclareMathOperator{\Ran}{Ran}
\DeclareMathOperator{\Rank}{rank}

\title[Eigenvalue computation for the $p$-Laplacian]{Eigenvalues of the discrete $p$-Laplacian via graph surgery} 



\author[G.~Berkolaiko]{Gregory Berkolaiko}
\author[M.~Hofmann]{Matthias Hofmann}

\address{Department of Mathematics, Texas A\&M University, College Station, USA}
\address{Fakultät Mathematik und Informatik, Fernuniversität Hagen, Hagen, Germany}
\email{gberkolaiko@tamu.edu}
\email{matthias.hofmann@fernuni-hagen.de}

\thanks{\emph{Acknowledgements.} M. Hofmann was supported by the Funda\c{c}\~ao para a Ci\^encia e a Tecnologia (FCT), Portugal, within the scope of the projects Spectral Optimal Partitions: geometric and numerical analysis, reference \href{https://doi.org/10.54499/2023.13921.PEX                 }{2023.13921.PEX}.}

\begin{document}
\begin{abstract}
We develop a Hellmann--Feynman type perturbation theory for the discrete signed $p$-Laplacian and apply it to a parametrized perturbation by edge cuts.   We show that the
eigenvalues of the signed $p$-Laplacian  can be characterized as
citical values of the parameter-dependent eigenvalues of a simpler graph. 
\end{abstract}

\maketitle

\section{Introduction}
\label{sec:intro}

The graph $p$-Laplacian is a generalization of the widely known
discrete Laplacian operator which found important applications in
spectral clustering \cite{BuHe09,LHDN10} and semi-supervised machine
learning \cite{FLZZT21,UJT21,ZhSc06,SPH07}.  Clustering applications
encouraged studies of the eigenvector morphology, especially of the
nodal domains \cite{BuHe09, LHDN10}, as well as of Cheeger-type
estimates for the eigenvalues (see \cite{KeMu16}).  Yet even numerical
computation of the eigenvalues of the $p$-Laplacian on a simple graph
is a difficult problem as demonstrated in the clustering analysis for \cite{BuHe09} for the first nontrivial eigenpair and the heuristic approximation of eigenpairs via p-Orthogonality in \cite{LHDN10}.  It is known that a
graph with $|V|$ vertices will have \emph{at least} $|V|$ eigenvalues,
analytically characterized via the Ljusternik--Schnirelmann principle
\cite{BiRy08, Li90}.  It has been recently found that on tree graphs the
Ljusternik--Schnirelmann eigenvalues are the only eigenvalues
\cite{DPT23}.

Motivated by the progress achieved in \cite{DPT23} (see also \cite{Deidda23}), we study here a parametric surgery on a graph.  We remove
an arbitrary selection of the graph's edges\footnote{In view of the
  results of \cite{DPT23}, the complement of a spanning tree of the
  graph is a natural set of edges to remove.}, replacing them with
parameter-dependent vertex potentials \cite{BRS12}.  This allows us to characterize
eigenvalues of a given graph as \emph{critical values} of the
parameter-dependent eigenvalues of a simpler graph.

Since the surgery we consider introduces vertex potentials, the
natural setting for our study is the $p$-Schr\"odinger operator (i.e.\
the Laplacian plus a potential).  Furthermore, to account for all
critical points we find it necessary to consider signed graphs
\cite{GLZ22,Zas82}.  To obtain our main result, we study
the perturbation theory of the eigenvalues of the signed
$p$-Schr\"odinger operator, obtaining first variation formulas for
both the eigenvalues (used in the main result) and the eigenvectors
(it comes free with our method).

We remark that there are natural analogues of the graph $p$-Laplacian
on Euclidean domain and metric (quantum) graphs \cite{KeMu16,Maz22}
and we expect our results to hold in that setting too, similarly to
the situation with the linear Laplacian \cite{BBRS12, BKS12, BCCM22}.
The paper is structured as follows. In Section~\ref{sec:results} we
introduce our object of study and formulate our results.  In
Section~\ref{sec:perturbation} we prove sufficient and necessary
conditions for differentiability of eigenvalues and eigenfunctions
with respect to a parameter.  The result is a formula of
Hellmann--Feynman type for the first variation of the eigenvalue and
the eigenfunction.  Section~\ref{sec:cuttinggraphs} contains a proof
of our main result, Theorem~\ref{thm:mainresult}.  Several examples are investigated in
Section~\ref{sec:examples}, providing illustrations to our results and
their limitation.

\section{Our setting and the main results}
\label{sec:results}

Let $\DG = (\DV, \DE)$ be a finite connected undirected simple graph,
with the set of vertices $\DV$ and the symmetric set of directed
edges\footnote{More precisely, $uv \in \DE \Rightarrow vu\in\DE$ and
  $uu\not\in\DE$ for all $u,v \in \DV$.} $\DE$.  For a symmetric subset $\DE_0 \subset \DE$ define the following the notation
\begin{align}
  \label{Sset_def}
  &\mathcal S(\DE_0, X)
    := \left\{ \gamma \in X^{\DE_0}
    \colon \gamma_{uv} = \gamma_{vu} \ \forall {uv\in \DE_0} \right\},\\
  \label{eq:Aset_def}
  &\mathcal A(\DE_0, X) := \left\{ \alpha \in X^{\DE_0}
    \colon \alpha_{u v} = \alpha_{v u}^{-1}\ \forall {uv\in \DE_0} \right\}.
\end{align}
On the set of edges we introduce a positive (symmetric) measure
$\omega \in \mathcal S(\DE, \mathbb
R_{+} )$ as well as a choice of signs
$\sigma \in  \mathcal S( \DE, \{-1,1\})$.
A positive measure $\rho: \DV \to \mathbb R_+$ on the set of vertices defines the weighted $p$-norm ($p\geq 1$) on functions $f \in \R^\DV$,
\begin{equation}
  \label{eq:p-normalized}
  \|f\|_{p}:= \left( \sum_{u\in \DV} \rho_u |f_u|^p \right)^{1/p}.
\end{equation}
    
Let $p>1$ be fixed and define the shorthand
\begin{equation}
    \label{eq:phip_def}
    \phi_p(x) := |x|^{p-2} x.
\end{equation}
When applied to a vector, $\phi_p$ is understood to act
\emph{componentwise}.  The signed $p$-Laplacian operator acts on a
function $f: \DV \to \mathbb R$ by
\begin{equation}
  \left(\Delta_p^{\sigma, \omega} f \right)(u)
  := \sum_{v: uv\in \DE} \omega_{uv} 
 \ \phi_p\big(f_u- \sigma_{uv} f_v\big).
\end{equation}
When $p=2$, we recover the usual linear (signed) graph Laplacian.
The signed $p$-Schr\"odinger operator is obtained by adding the
\emph{potential} $\kappa: \DV\to \R$, 
\begin{equation}
  \Hp f = \Hp^{\sigma, \omega, \kappa} f
  := \Delta_p^{\sigma, \omega} f + \kappa \phi_p(f),
\end{equation}
where the multiplication by $\kappa$ is componentwise.
A note on our notational convention: occasionally we will be changing the
graph parameters $\kappa$, $\sigma$ and, to a lesser extent, $\omega$;
in those instances we will include $\sigma,\omega,\kappa$ in the
notation for the operator.

We say that $f$ is an eigenfunction of $\Hp$ if there exists
$\lambda \in \R$ such that
\begin{equation}\label{eq:peigprobsigma}
	\Hp f = \lambda \rho \phi_p(f).
\end{equation}
We refer to $(\lambda, f)$ as the \emph{eigenpair} of
$\Hp$.  We usually assume the eigenfunctions to be
\emph{normalized}: $\|f\|_{p}=1$.

Our principal tool is the parametric dependence of a given eigenpair
on a set of parameters $\alpha = (\alpha_1, \ldots, \alpha_n)$ which
enter the edge weights $\omega$ and the potential $\kappa$.  To
approach this perturbation problem, we recall that the normalized
eigenfunctions can be identified as the critical points of the
$p$-form
\begin{equation}
  \frac12\sum_{uv\in \DE}  \omega_{uv} \big|f_u- \sigma_{uv}
  f_v\big|^p + \sum_{u\in \DV} \kappa_u \big|f_u\big|^p
\end{equation}
restricted to the unit ball $\|f\|_{p} = 1$.  The corresponding
Lagrangian
$L(\cdot, \cdot; \alpha): \mathbb R \times \mathbb R^{\DV} \to \mathbb
R$ is defined via
\begin{equation}
  \label{eq:lagrangian}
  L(\lambda, f; \alpha) 
  = \frac12\sum_{uv \in \DE} \omega_{uv}(\alpha) 
  \big|f_u - \sigma_{uv} f_v\big|^p 
  + \sum_{u\in \DV} \kappa_u(\alpha) |f_u|^p
  - \lambda \left(\sum_{u\in V} \rho_u |f_u|^p -1 \right),
\end{equation}
where we have included the dependence of $\omega$ and $\kappa$ on the
parameter $\alpha \in \R^n$.  Under the assumption of smoothness of
$L$ at the point in question (which is automatic if, for example, $p$
is an even integer), the critical points
$(\lambda, f)\in \R \times \R^{\DV}$ of \eqref{eq:lagrangian}
characterize solutions of \eqref{eq:peigprobsigma} satisfying
$\|f\|_p=1$.  Introducing $F: \R \times \R^\DV \to \R^\DV$,
\begin{equation}
  \label{eq:F_def}
  F(\lambda, f; \alpha):= \nabla_f L( \lambda, f; \alpha),
\end{equation}
the condition for the critical point is $\nabla_{\lambda, f}
L(\lambda, f; \alpha) =0$, or equivalently,
\begin{equation}
  \label{eq:critical_point_condition}
  \|f\|_{p}^p =1,
  \qquad
  F(\lambda, f; \alpha) = 0.
\end{equation}

We now formulate a sufficient condition that guarantees smoothness of
the eigenpair with respect to the parameter $\alpha$; it is somewhat
analogous to the condition of simplicity of the eigenvalue in the
linear ($p=2$) case. By $\nabla_f^2 L$ we denote the
Hessian of $L$ with respect to the $f$ variable (a $|\DV|\times |\DV|$
symmetric matrix). 

\begin{remark}
  \label{rem:del_notation}
  We understand $\nabla X$ as a (column) vector and $D X$ as a functional (row vector); $\nabla^2 X$ is used to denote the Hessian in a slight (and traditional) abuse of notation.  With the notation just introduced, $D_fF = \nabla_f^2 L$.
\end{remark}

\begin{definition}
  \label{def:reg_eigen}
  Suppose $(\lambda_0, f_0)$ is a normalized eigenpair of $\Hp$
  (i.e.~a critical point of $L(\lambda, f)$). We say that it
  is \emph{regular} if $\nabla_f^2 L(\lambda_0, f_0)$ exists
  and
  \begin{equation}
    \label{eq:eig_continuation_cond}
    \dim \ker\left(\nabla_f^2 L(\lambda_0, f_0)\right) =1.
  \end{equation}
\end{definition}

We remark that as a consequence of the scale
invariance of the eigenvalue equation \eqref{eq:peigprobsigma}, the
kernel of $\nabla_f^2 L(\lambda_0, f_0)$ is \emph{at least}
1-dimensional, see Lemma~\ref{lem:kernel}.  Thus, the regularity
condition requires the Hessian to have minimal kernel.  We also stress
that regularity is a property of an eigenpair of a particular
operator, rather than a parametric family; for this reason we
omitted $\alpha$ in the objects referred to in Definition~\ref{def:reg_eigen}.

\begin{remark}
  \label{rem:reg_simple}
  The regularity condition is an analogue of eigenvalue simplicity in
  the linear case. Namely, if $p=2$, then
  $\nabla^2_f L(\lambda, f) = \mathcal{H}_2 -
  \lambda I $ and
  \begin{equation}
    \dim \ker\left(\nabla^2_f L(\lambda, f)\right) 
    = \operatorname{mult}(\lambda),
  \end{equation}
  that is, an eigenpair is regular if and only if the associated eigenvalue is simple.  Limitations to this intuition are explores in Section~\ref{sec:regularity}.
\end{remark}

For a regular eigenpair, we can compute the derivatives of
$\lambda(\alpha)$ and $f(\alpha)$ using formulas reminiscent of the
classical linear Hellmann--Feynman formulas.

\begin{theorem}
  \label{thm:derivationformula}
  Suppose $(\lambda_0, f_0)$ is a regular
  eigenpair of $\Hpaa$ and suppose that $F$ is differentiable in
  $\alpha$ at $\alpha_0$.  Then for
  $\alpha=(\alpha_1, \ldots, \alpha_n)$ in a neighborhood of
  $\alpha_0$, there exists a differentiable at $\alpha = \alpha_0$
  family $\big(\lambda(\alpha), f(\alpha)\big)$ of critical points of
  \eqref{eq:lagrangian} with $\|f(\alpha)\|_p=1$ and
  \begin{align}
    \label{eq:eval1variation}
    \frac{\mathrm d\lambda}{\mathrm d\alpha_i} 
    &= \left< \frac{1}{p} D_{\alpha_i} F,\, f_0 \right>,
    \\
    \label{eq:evec1variation}
    \frac{\mathrm d f}{\mathrm d\alpha_i}
    &= -T \left(T^*\, \nabla_f^2 L\, T\right)^{-1} T^* D_{\alpha_i} F,
  \end{align}
  where $D_{\alpha_i} F$ and $\nabla_f^2 L$ are evaluated at
  $(\lambda_0, f_0; \alpha_0)$ and $T$ is a $|V|\times(|V|-1)$ matrix
  whose columns form a basis of the space
  $\Span\big\{\rho \phi_p(f_0)\big\}^\perp$.
\end{theorem}

\begin{remark}\label{rem:kernelcondition}
  We will see that $\Span\big\{\rho\phi_p(f_0)\big\}^\perp$ is transversal
  to $\Span\{f_0\}$, which, by the regularity condition
  \eqref{eq:eig_continuation_cond}, coincides with
  $\ker\left(\nabla_f^2 L(\lambda_0, f_0; \alpha_0)\right)$.
  Therefore, $T^*\, \nabla_f^2 L\, T$ is invertible and
  \eqref{eq:evec1variation} is well-defined.  Different valid choices
  of $T$ are related by $T' = T C$, where $C$ is $(V-1)\times(V-1)$
  invertible, therefore the result of equation
  \eqref{eq:evec1variation} is independent of the choice of $T$.
\end{remark}

\begin{remark}
    \label{rem:pseudoinverse}
    The matrix $T \left(T^*\, \nabla_f^2 L\, T\right)^{-1} T^*$ is a
    reflexive generalized inverse of the matrix $\nabla_f^2 L$ (or
    ``$\{1,2\}$-inverse'' in terminology of \cite{BeGr03}).  If
    $p=2$, then $\phi_p(f_0)=f_0$, and, if the column vectors of $T$
    are orthonormal, the matrix
    $T \left(T^*\, \nabla_f^2 L\, T\right)^{-1} T^*$ coincides with
    the Moore-Penrose pseudoinverse $\left(\nabla_f^2 L\right)^+ = \left(\mathcal{H}_2-\lambda I\right)^+$, also known as the \emph{reduced resolvent}.
\end{remark}

An application of Theorem~\ref{thm:derivationformula} and our main
result is a characterization of eigenvalues of $\Hp^{\sigma, \omega}$
as the critical points of the eigenvalues of a parametric family of
simpler graphs.  We now define those graphs.

Let 
\begin{equation}
    \label{eq:Ecut_def}
    \CE=\{u_1 v_1,\, v_1u_1,\, \ldots,\, u_n v_n,\, v_n u_n\}
\end{equation}
be a symmetric subset of the edge set $\DE$.  We obtain a \emph{cut
  graph} $\CG = (\DV, \DE\setminus \CE) $ by removing the edges $\CE$
from $\DG$.  Introducing the \emph{cut parameters} $\alpha \in
\mathcal{A}(\CE, \mathbb R\setminus \{0\})$, see \eqref{eq:Aset_def},
we define the Hamiltonian $\hHpa^{\sigma, \omega}$ on $\CG$ by
\begin{equation}
  \label{eq:Hpa}
  \left(\hHpa^{\hat \sigma, \hat\omega, \hat \kappa}  f\right)(u)
  := \left(\Delta_p^{\hat \sigma, \hat \omega } f\right)(u) + \hat \kappa_u(\alpha) \phi_p\big(f_u\big), 
\end{equation}
where
\begin{align}
  \label{eq:hatsigma_def}
  &\hat \sigma := \sigma|_{\DE\setminus \CE}, \\
  \label{eq:hatomega_def}
  &\hat \omega := \omega|_{\DE \setminus \CE}, \\
  \label{eq:hatkappa_def}
  &\hat \kappa_u(\alpha)
  := \kappa_u + \sum_{u\in \DV:\, uv \in \CE }
  \omega_{uv}\phi_p(1- \alpha_{uv})  \quad \forall u \in \DV.
\end{align}

The set $\mathcal{A}(\CE, \mathbb R\setminus \{0\})$ of admissible cut
parameters is parametrized by $\alpha_1, \ldots, \alpha_n \in \mathbb R\setminus \{0\}$ via the convention
\begin{equation}
  \label{eq:parameterization}
  \alpha_{u_1 v_1} = \alpha_1, \quad
  \alpha_{v_1u_1} =\alpha_1^{-1}, \quad
  \cdots \quad,
  \alpha_{u_n v_n} = \alpha_n, \quad
  \alpha_{v_nu_n} =\alpha_n^{-1}.
\end{equation}
We remark that $\alpha_i =1$ for some $i=1,\ldots, n$ will be
sometimes excluded from our statements.

We can now formulate our main result.
\begin{theorem}
  \label{thm:mainresult}
  If $(\lambda_0, f_0)$ is an eigenpair of $\Hp^{\sigma, \omega, \kappa}$
  satisfying
  \begin{equation}
    \label{eq:mainresultdich}
    f_0(u) = 0 \ \Leftrightarrow\ f_0(v) =0
    \qquad \text{for any}\quad uv \in \CE,
  \end{equation}
  then there exists
  $\alpha_0 \in \mathcal A(E_{cut}, \mathbb R\setminus \{ 0\})$ such
  that $(\lambda, f)$ is an eigenpair of
  $\hHpaa^{\hat\sigma, \hat\omega, \hat\kappa}$. If $(\lambda, f)$ is
  regular as an eigenpair of
  $\hHpaa^{\hat\sigma, \hat\omega, \hat\kappa}$ and the corresponding
  $F$ is differentiable in $\alpha$ at $\alpha_0$, then
  \begin{equation}
    \label{eq:crit_pnt_conclusion}
    \nabla_{\alpha}\lambda |_{\alpha= \alpha_0} =0.    
  \end{equation}
  
  Conversely, let $(\lambda_0, f_0)$ be a regular eigenpair of
  $\hHpaa^{\hat \sigma, \hat\omega, \hat\kappa}$ with
  $\hat\kappa(\alpha)$ having the form \eqref{eq:hatkappa_def} and
  with $\alpha_0 \in \mathcal A(\CE, \R\setminus \{0,1\})$.  If
  $\nabla_\alpha \lambda |_{\alpha = \alpha_0} = 0$, then
  $(\lambda_0, f_0)$ is an eigenpair of $\Hp^{\sigma, \omega, \kappa}$
  for a choice of $\sigma\in\mathcal{A}(\DE, \{-1,1\})$ extending
  $\hat\sigma$, i.e.
  $\sigma \big |_{\DE \setminus \CE} = \hat \sigma$.
\end{theorem}

\begin{remark}
  The condition of differentiability of $F$ in $\alpha$ can be made
  more precise, by assuming that either $p>2$ or all $\alpha_i\neq 1$ (which is
  anyway assumed in the converse part of the Theorem).
\end{remark}

\begin{remark}
  In the second part of Theorem~\ref{thm:mainresult}, the weights
  $\omega$ and the potential $\kappa$ on the full graph are directly
  implied by the form of $\hat\kappa$,
  equation~\eqref{eq:hatkappa_def}, as well as the condition
  $\omega|_{\DE\setminus\CE} = \hat\omega$.  The signs $\sigma$, on
  the other hand, will also depend on $f_0$ and $\alpha_0$.
\end{remark}

\begin{remark}
    We will see in the proof of Theorem~\ref{thm:mainresult}
    (see Section~\ref{sec:cuttinggraphs}) that the
    conditions $\alpha_0 \in \mathcal A(\CE, \mathbb R \setminus \{
    0,1\})$ and $\nabla_\alpha \lambda |_{\alpha=\alpha_0} =0$ imply
    the property \eqref{eq:mainresultdich} for an
    eigenpair $(\lambda_0, f_0)$ of $\hHpaa^{\hat \sigma,\omega}$.
\end{remark}

\begin{remark}\label{rmk:mainresult}
  Theorem~\ref{thm:mainresult} can be strengthened when $\CE$ does not
  contain cycles since in this case
  $\alpha \in \mathcal A(\CE, \mathbb R\setminus \{ 0 \})$ can be
  shown to be unique if $f_u \neq 0 \neq f_v$ for all $uv \in \CE$.
  If $f_u = f_v =0$ for some $uv \in \CE$, then $\alpha_{uv}\neq 0$
  can be chosen arbitrary.
\end{remark}

Theorem \ref{thm:mainresult} may be used as follows: choose a spanning tree of the graph $\DG$ and let $\CE$ be the complementary set of edges.  The cut graph $\CG$ is now a tree and by \cite{DPT23} we know that $p$-Laplacian on a tree has exactly $|\DV|$ eigenvalues.
For every $\alpha$, the eigenvalue problem may now be reduced to a single equation by recursively solving the system from the leaves in.  The eigenvalues of the original graph may now be identified as critical points of the tree eigenvalues as functions of $\alpha$.  This proposal is implemented in Section\ref{sec:examples} for two simple graphs.

\section{First-order perturbation theory for the \texorpdfstring{$p$}{p}-Laplacian.} 
\label{sec:perturbation}

In this section we establish Theorem~\ref{thm:derivationformula}, i.e.\ develop a first-order (Hellmann--Feynman) perturbation theory for the $p$-Laplacian.
We start with an auxiliary lemma which, in particular, sheds light on the regularity condition \eqref{eq:eig_continuation_cond}.

\begin{lemma}
    \label{lem:kernel}
    Assume $D_f F(\lambda, f; \alpha)$ exists.  Then
    \begin{equation}
        \label{eq:DfF}
        \big(D_f F(\lambda, f; \alpha)\big) f  
        = (p-1) F(\lambda, f; \alpha).
    \end{equation}
    In particular, at a critical point $(\lambda_0, f_0)$ of
    $L(\cdot,\cdot;\alpha_0)$,
    \begin{equation}\label{eq:kernel}
        f_0 \in \ker\left(\nabla_{f}^2 L(\lambda_0, f_0; \alpha_0)\right).  
    \end{equation}
\end{lemma}

\begin{remark}\label{rem:derivationformula}
  We now see that condition \eqref{eq:eig_continuation_cond} in the
  definition of a \emph{regular} eigenpair is requiring that the
  dimension of $\ker\left(\nabla_{f}^2 L\right)$ is the smallest
  possible.
\end{remark}

\begin{proof}[Proof of Lemma~\ref{lem:kernel}]
  For $t>0$ we have the scaling invariance
  \begin{equation}
    F(\lambda, t f; \alpha) = t^{p-1}  F(\lambda, f; \alpha).
  \end{equation}
  Differentiating with respect to $t$, we have
  \begin{equation}\label{eq:homogeneousderivative}
    \begin{aligned}
      \big(D_f F(\lambda, f; \alpha)\big) f
      &= \frac{\mathrm d}{\mathrm dt} F( \lambda, tf; \alpha)  \big |_{t=1} \\
      &= (p-1) t^{p-2}|_{t=1} F(\lambda, f; \alpha)= (p-1) F(\lambda, f; \alpha). 
    \end{aligned}
  \end{equation}
    At a critical point $(\lambda_0, f_0)$ of $L(\cdot, \cdot; \alpha_0)$, $F(\lambda_0, f_0; \alpha_0)=0$ and \eqref{eq:kernel} follows immediately.
\end{proof}

\begin{proof}[Proof of Theorem~\ref{thm:derivationformula}]
    Since $ f_0^*\, \rho\phip(f_0) = \| f_0\|_p^p =1$, we have 
    \begin{equation}\label{eq:proofnontrivialintersect}
        f_0 \not \in \Span\{\rho\phip(f_0)\}^\perp= \operatorname{Ran}(T).
    \end{equation} 
    Therefore the range of the $|\DV| \times |\DV|$ matrix $(f_0\ T)$ is $\mathbb R^{\DV}$ and the $(1+|V|)\times(1+|V|)$ matrix
    \begin{equation}
        S := \begin{pmatrix} 1 & 0 & 0 \\
        0 & f_0 & T \end{pmatrix}
    \end{equation}
    is invertible. 

    Formally, using the inverse function theorem on
    $\nabla_{\lambda, f} L=0$ (see also
    \eqref{eq:critical_point_condition}) we have
    \begin{equation}\label{eq:inversefunction}
        \begin{split}
          \frac{\mathrm d}{\mathrm d\alpha_{i}}
          \begin{pmatrix} \lambda \\ f \end{pmatrix}
          &=  - \left( \frac{1}{p} \nabla^2_{\lambda, f} L
          \right)^{-1}
          \frac{1}{p} D_{\alpha_i} \nabla_{\lambda, f} L \\
        &= -S \left( S^*\, \frac{1}{p} \nabla^2_{\lambda, f} L\, S
        \right)^{-1} S^*
        \begin{pmatrix} 0 \\\frac{1}{p} D_{\alpha_i} F \end{pmatrix}.
      \end{split}
    \end{equation}
    In order to use the inverse function theorem we need to show that
    $S^* \,\frac{1}{p} \nabla_{\lambda, f}^2 L\, S$ is invertible
    under our assumptions.
    
    From Lemma~\ref{lem:kernel} and regularity of $(\lambda_0, f_0)$
    as an eigenpair of $\Hpaa$, see Definition~\ref{def:reg_eigen}, we
    conclude
    \begin{equation}
        \ker(\nabla_f^2 L)= \Span \{ f_0\}.
    \end{equation}
    Then, due to \eqref{eq:proofnontrivialintersect}, we have
    \begin{equation}
      \label{eq:RanKer1}
      \operatorname{Ran}(T) \cap \ker(\nabla_f^2 L ) = \{ 0 \}.
    \end{equation}
    Since, by dimension counting,
    $\operatorname{Ran}(T) + \ker(\nabla_f^2 L) = \mathbb R^\DV$
    we also have
    \begin{equation}
      \label{eq:RanKer2}
      \ker(T^*) \cap \Ran(\nabla_f^2 L) = \left ( \Ran(T) + \ker(\nabla_f^2 L)  \right )^\perp = \{ 0 \}.
    \end{equation}

    From \eqref{eq:RanKer1} and \eqref{eq:RanKer2} we have, for any
    vector $g \in \mathbb{R}^{|V|-1}$,
    \begin{equation}
      \label{eq:THessT}
      \HT g = 0\quad \Rightarrow\quad
      \nabla_f^2 L \, T g = 0 \quad \Rightarrow\quad
      Tg = 0 \quad \Rightarrow\quad
      g = 0,
    \end{equation}
    the last implication being true because, by its definition, $T$ is
    injective.  We thus conclude that $\HT$ is invertible. Using
    \eqref{eq:kernel}, as well as $T^* \phi_p(f_0) =0$ and
    $f_0^* \phi_p(f_0)=\|f_0\|_p^p=1$, we evaluate
    \begin{align}
        \HS &= 
        \begin{pmatrix} 1 & 0 \\ 0 & f_0^*\\ 0 & T^*\end{pmatrix} \begin{pmatrix} 0 & -\rho\phip(f_0)^* \\ -\rho\phip(f_0) & \frac1p\nabla_f^2 L \end{pmatrix} \begin{pmatrix} 1 & 0 & 0 \\ 0 & f_0 & T \end{pmatrix} \\
        &= \begin{pmatrix} 0 & -1 & 0 \\ -1 & 0 & 0 \\ 0 & 0& \frac{1}{p} \HT  \end{pmatrix},
    \end{align}
    which is manifestly invertible. Then from \eqref{eq:inversefunction} we compute
  \begin{align}
    \frac{d}{d{\alpha_i}} \begin{pmatrix} \lambda \\ f \end{pmatrix} &= -\begin{pmatrix} 1 & 0 & 0 \\ 0 & f_0 & T \end{pmatrix} \begin{pmatrix} 0 & -1 & 0 \\ -1 & 0 & 0 \\ 0 & 0& \left(\frac{1}{p} \HT\right)^{-1}  \end{pmatrix} \begin{pmatrix} 1 & 0 \\ 0 & f_0^*\\ 0 & T^*\end{pmatrix}   \begin{pmatrix} 0 \\\frac{1}{p} D_{\alpha_i} F \end{pmatrix} \\
    &= \begin{pmatrix} \big \langle \frac{1}{p}D_{\alpha_{i}} F, f_0 \big \rangle \\   -T \left(\HT\right)^{-1} T^* D_{\alpha_{i}} F  \end{pmatrix} .
  \end{align}
\end{proof}

In the next theorem, we deduce the formula for the derivative of the eigenvalue $\lambda$ --- equation \eqref{eq:eval1variation} of Theorem~\ref{thm:derivationformula} --- but under a slightly different set of conditions.  Namely, we assume a priori that there is a smooth branch of eigenpairs $(\lambda(\alpha), f(\alpha))$.  That this can happen even if the regularity assumption in Theorem~\ref{thm:derivationformula} does not hold is illustrated by the Kato--Rellich linear perturbation theory in one parameter (also known as the Kato Selection Theorem), where smooth continuation of eigenvalue curves exists even when the curves intersect. 

\begin{theorem}\label{thm:necessarycond}
  Suppose $\left(\lambda(\alpha), f(\alpha)\right)$ is a locally smooth branch
  of normalized eigenpairs of $\Hpa$ and suppose that the
  corresponding
  $F(\cdot,\cdot;\cdot): \R \times \R^V \times \R^n \to \R^V$ is
  differentiable in $\alpha_i$ and in $f$ at
  $\lambda=\lambda(\alpha)$, $f=f(\alpha)$.
  Then
    \begin{equation}
        \label{eq:derivgen}
         \frac{\mathrm d\lambda}{\mathrm d \alpha_{i}}
         = \left \langle \frac{1}{p} D_{\alpha_{i}} F, f \right \rangle.
    \end{equation}
\end{theorem}

\begin{remark}\label{rmk:BiMu}
  In \cite[Lemma~A.1]{BiMu23} a similar formula was independently
  derived in the case of varying potentials $\kappa(\alpha)$ for such
  a priori smooth branches of normalized eigenpairs
  $(\lambda(\alpha), f(\alpha))$ for $p$-Schr\"odinger
  operators. Namely, it was shown that
  \begin{equation}
    \frac{\mathrm d\lambda}{\mathrm d\alpha}
    = \sum_{v\in \DV} \frac{\mathrm d \kappa_v}{\mathrm d\alpha}|f_v|^p.
  \end{equation}
\end{remark}

\begin{proof}[Proof of Theorem~\ref{thm:necessarycond}]
  By assumption, $F\big(\lambda(\alpha), f(\alpha); \alpha\big)=0$, therefore
  \begin{align*}
    0
    &= \frac{\mathrm d}{\mathrm d \alpha_i}
      \left< F(\alpha, \lambda, f), f \right> \\
    &= \langle D_{\alpha_i} F, f\rangle 
      + \langle D_\lambda F, f\rangle
      \frac{\mathrm d \lambda}{\mathrm d\alpha_i} 
      + \left \langle D_f F \;\frac{\mathrm df}{\mathrm d\alpha_i},
      f\right \rangle 
      + \left \langle F, \frac{\mathrm df}{\mathrm d\alpha_i}\right \rangle.
  \end{align*}
  The last term is zero by the assumption
  $F=F\big(\lambda(\alpha), f(\alpha); \alpha\big)=0$.  
  In the third term, we use self-adjointness of the Hessian
  $D_fF = \nabla_f^2L$ and Lemma~\ref{lem:kernel} to simplify
  \begin{equation}
    \left< D_f F \; \frac{\mathrm d f}{d\alpha_i},\, f \right> 
    = \left< \frac{\mathrm d f}{d\alpha_i},\, D_f F \; f\right>
    = (p-1) \left< \frac{\mathrm d f}{d\alpha_i},\, F\right> = 0.
  \end{equation}
  The second
  term can be simplified using \eqref{eq:lagrangian} and the
  normalization of $f$ as follows:
  \begin{equation}
    \label{eq:Dlambda_f}
    \left< D_\lambda F (\lambda, f; \alpha), f \right>
    = \left< D_\lambda \nabla_f L(\lambda, f; \alpha), f \right>
    = \left< \nabla_f \left(1-\|f\|_p^p\right), f \right>
    = -p \|f\|_p^p = -p.
  \end{equation}
  To summarize, we arrive to
  \begin{equation}\label{eq:deriv}
    \frac{\mathrm d\lambda}{\mathrm d \alpha_{i}}
    =\left< \frac{1}{p} D_{\alpha_i} F\big(\lambda(\alpha), f(\alpha);
      \alpha\big),\, f(\alpha) \right>,
  \end{equation}    
    completing the proof.
\end{proof}

\section{Cutting graphs}\label{sec:cuttinggraphs}

We now apply the perturbation theory developed in
Theorem~\ref{thm:derivationformula} and Section~\ref{sec:perturbation}
to the Hamiltonian $\hHpa^{\hat\sigma,\hat\omega,\hat\kappa}$ of the
cut graph, equations \eqref{eq:Hpa}-\eqref{eq:hatkappa_def}.

\begin{proof}[Proof of Theorem~\ref{thm:mainresult}]
  Given an eigenpair $(\lambda_0, f_0)$ of
  the Hamiltonian $\Hp^{\sigma, \omega, \kappa}$, we can construct
  $\alpha_0 \in \mathcal A\big(\CE, \mathbb R\setminus \{ 0 \}\big)$
  via
  \begin{equation}
    \label{eq:alpha0_def}
    \alpha_{uv} :=
    \begin{cases}
      \frac{\sigma_{uv} f_v}{f_u}, & f_u \neq 0 \\
      -1, & f_u = 0.
    \end{cases}
  \end{equation}
  for $uv\in \CE$.  Note that $\alpha_0$
  satisfies $\alpha_{uv} = \alpha_{vu}^{-1}$ because
  $\sigma_{uv} = \sigma_{vu}^{-1}$ and, in the second case of
  equation~\eqref{eq:alpha0_def}, due to
  condition~\eqref{eq:mainresultdich}.

  By construction, we have for all $uv\in \CE$
  \begin{equation}
    \label{eq:proofconst}
    \phip(f_{u} - \sigma_{u v} f_{v})
    = \phip(f_{u} - \alpha_{uv} f_u)
    = \phip(1- \alpha_{u v}) \phip( f_{u}).
  \end{equation}
  Recalling the definition of $\hat\kappa$, equation  \eqref{eq:hatkappa_def}, we evaluate, for any $u\in\DV$,
  \begin{align}
    \nonumber
   \left(\hHpaa^{\hat\sigma, \hat\omega, \hat\kappa} f_0\right)(u)
    &= \sum_{uv\in \DE \setminus \CE} \hat\omega_{uv}
      \phip(f_u - \hat \sigma_{uv} f_v) + 
      \hat\kappa_u \phip(f_u) \\
    \nonumber
    &= \sum_{uv\in \DE \setminus \CE} \omega_{uv}
      \phip(f_u - \sigma_{uv} f_v) + 
      \left(\kappa_u
      + \sum_{uv \in \CE} \omega_{uv} \phip(1- \alpha_{uv})\right)
      \phip(f_u) \\
    \label{eq:H_coincide}
    &= \sum_{uv\in \DE} \omega_{uv}
      \phip(f_u - \sigma_{uv} f_v) + \kappa_u \phip(f_u)
    \\ \nonumber
    &= \left(\Hp^{\sigma, \omega, \kappa} f_0\right)(u)
      = \lambda_0 \rho \phip\big(f_u\big),
  \end{align}
  i.e.\ 
  $(\lambda_0, f_0)$ is an eigenpair of $\hHpaa^{\hat \sigma,
    \hat\omega, \hat\kappa}$.  We now evaluate
  $\nabla_{\alpha} \lambda |_{\alpha=\alpha_0}$ by
  Theorem~\ref{thm:derivationformula}.
  We compute explicitly
  \begin{equation}\label{eq:deriv2}
    \frac{1}{p} D_{\alpha_i} F_v(\lambda, f; \alpha)
    =
    \begin{cases}
      \omega_{u_i v_i} (p-1) |1-\alpha_i^{-1}|^{p-2}
      \frac{\phip(f_{u_i})}{\alpha_i^2},
      &\quad v=u_i, \\
      -\omega_{u_i v_i} (p-1) |1- \alpha_i|^{p-2} \phip(f_{v_i}),
      &\quad v=v_i, \\ 0,
      &\quad \text{otherwise,}
    \end{cases}
  \end{equation}
  and then, using equation~\eqref{eq:eval1variation},
  \begin{align}
    \frac{\mathrm d\lambda}{\mathrm d\alpha_i}(\alpha_0)
    \label{eq:derivvary}
    &= (p-1) \omega_{u_i v_i} |1-\alpha_i|^{p-2}
      \left( \frac{|f_{u_i}|^p}{|\alpha_i|^p} - |f_{v_i}|^p \right) = 0,
  \end{align}
  where in the last step we used equation~\eqref{eq:alpha0_def}.

  For the converse part of the Theorem, assume $\nabla_{\alpha}
  \lambda |_{\alpha=\alpha_0} = 0$.  We apply
  Theorem~\ref{thm:derivationformula}, where differentiability in
  $\alpha$ follows from $\alpha_i \neq 1$.  Getting to
  \eqref{eq:derivvary} and using $\alpha_i \neq 1$ again, we conclude
  \begin{equation}
    \label{eq:criticalimply}
    f_{v_i} = \pm \alpha_{i} f_{u_i},
  \end{equation}
  and, in particular, $f_{u_i}=0$ if and only if $f_{v_i}=0$.  We now let
  \begin{equation}
    \sigma_{uv} := \begin{cases}
      \hat \sigma_{uv}, &\quad uv\in \DE \setminus \CE,\\
      \frac{\alpha_{uv} f_u}{f_v}, &\quad uv\in \CE, f_v \neq 0, \\
      1,&\quad uv \in \CE, f_v =0.
    \end{cases}
 \end{equation}
 This choice of $\sigma$ ensures \eqref{eq:alpha0_def} holds; we now
 reverse the steps of equation~\eqref{eq:H_coincide} to get
 \begin{equation*}
   \left(\Hp^{\sigma, \omega, \kappa} f_0\right)(u)
   = \left(\hHpaa^{\hat\sigma, \hat\omega, \hat\kappa} f_0\right)(u)
   = \lambda_0 \rho \phip\big(f_u\big).
 \end{equation*}
\end{proof}

 \section{Examples}
 \label{sec:examples}

\subsection{Regularity condition}
\label{sec:regularity}

In the next two examples we explore the differences between the geometric simplicity and regularity (Definition~\ref{def:reg_eigen}) of an eigenpair.

\begin{example}[Regular need not be simple]
  \label{ex:nonsimpleregular}
  In this example we show that a regular eigenpair need not satisfy the (intuitive) notion of geometric simplicity: a single eigenvector
  up to an overall constant factor.

  Consider the triangle graph shown in Figure~\ref{fig:triangle},
  which we will consider with $p=2$ and $p=4$.  Assuming
  $\rho\equiv 1$, $\omega\equiv 1$, $\sigma \equiv 1$, and
  $\kappa\equiv 0$, we have the eigenvalue equations
  \begin{equation}
    \label{eq:eig_eq_triangle}
    \begin{cases}
      (f_1-f_2)^{p-1} + (f_1-f_3)^{p-1} = \lambda f_1^{p-1} \\
      (f_2-f_1)^{p-1} + (f_2-f_3)^{p-1} = \lambda f_2^{p-1} \\
      (f_3-f_1)^{p-1} + (f_3-f_2)^{p-1} = \lambda f_3^{p-1}.
    \end{cases}
  \end{equation}
  A direct check verifies the $\lambda = 1 + 2^{p-1}$ is an eigenvalue
  with example eigenvectors $f=(1,-1,0)$, $f=(0, 1, -1)$, and $f=(-1, 0,
  1)$.  For the linear operator ($p=2$), these vectors span a
  2-dimensional eigenspace: one vector is a linear combination of the
  other two.  In the nonlinear case ($p=4$) the three eigenvectors
  can be shown to be isolated (see also Section~\ref{sec:triangle})
  modulo an overall constant.

  To check whether the eigenpair $\lambda = 1 + 2^{p-1}$, $f=(1,-1,0)$
  is regular, we differentiate the Lagrangian
  \begin{equation}
    L(\lambda, f; \alpha) = |f_1-f_2|^p + |f_1-f_3|^p + |f_2-f_3|^p
    + \lambda (1-\|f\|^p_p),
  \end{equation}
  obtaining
  \begin{equation}
    \frac{1}{p(p-1)} \nabla_{f}^2 L(\lambda, f)
    =
    \begin{pmatrix}
      F_1 & -|f_1-f_2|^{p-2} & -|f_1-f_3|^{p-2} \\
      - |f_2-f_1|^{p-2} & F_2 & - |f_2-f_3|^{p-2} \\
      - |f_3-f_1|^{p-2} & - |f_3-f_2|^{p-2} & F_3
    \end{pmatrix},
  \end{equation}
  where
  \begin{align}
    F_1 &:= |f_1-f_2|^{p-2} + |f_1-f_3|^{p-2} - \lambda |f_1|^{p-2} \\
    F_2 &:= |f_2-f_1|^{p-2} + |f_2-f_3|^{p-2} - \lambda |f_2|^{p-2} \\
    F_3 &:= |f_3-f_1|^{p-2} + |f_3-f_2|^{p-2} - \lambda |f_3|^{p-2}.
  \end{align}
  For $p=2$ we get
  \begin{equation}
    \label{eq:Hess_p2}
    \nabla_{\lambda, f}^2 L(3, (1, -1, 0))
    = 2
    \begin{pmatrix}
      -1& -1 & -1 \\ -1 & -1 & -1 \\ -1& -1 & -1
    \end{pmatrix}.
  \end{equation}
  The rank is 2 and the eigenpair is not regular, in agreement with
  Remark~\ref{rem:reg_simple}.

  For $p=4$, the Hessian is of rank 1:
  \begin{equation}
    \nabla_{\lambda, f}^2 L(9, (1, -1, 0))
    = 12
    \begin{pmatrix}
      -4& -4 & -1 \\ -4 & -4 & -1 \\ -1& -1 & 2
    \end{pmatrix}.
  \end{equation}
  We observe that the eigenpair $\lambda = 1 + 2^{p-1}$, $f=(1,-1,0)$
  is regular, despite not being ``geometrically simple'' in the intuitive way.
\end{example}

\begin{remark}
  By Morse Lemma, the regularity condition does imply that the eigenvector is isolated.
\end{remark}

\begin{example}[Simple need not be regular]
  \label{ex:nonregularsimple}
  For an arbitrary connected graph $\DG =(\DV, \DE)$, consider $\Delta_p^{\sigma, \omega}$ with $\sigma \equiv 1$. Then the critical points of
  \begin{equation}
    L(\lambda, f; \alpha) =  \sum_{u\in \DV:\,  uv \in \DE} \omega_{uv}(\alpha) 
    \big|f_u - f_v\big|^p - \lambda \left(\sum_{u\in V} \rho_u |f_u|^p -1 \right).
  \end{equation}
  are the normalized eigenpairs of $\Delta_p^{\omega}$.

  Since the potential $\kappa \equiv 0$, $\mathbf1 = (1,1,\ldots, 1)$ is an eigenvector of $ \Delta_p^\omega$ corresponding to the eigenvalue $\lambda=0$.  The eigenvalue equation is
\begin{equation}\label{eq:averagingproperty}
    \Delta_p f(u) = \sum_{v\in \DV: uv \in \DE} \omega_{uv} \phi_p(f(u) - f(v))=0
\end{equation}
for all $u\in \DV$. Since $\phi_p \geq 0$, we must have $\phi_p(f(u) - f(v))=0$, which implies $f(u)=f(v)$
for all edges $uv\in \DE$.  Since the graph is connected, we obtain that $f$ must be constant, and therefore $\mathbf1$ is a only eigenvector corresponding to the eigenvalue $\lambda=0$. 
 
For $1\le p < 2$ one easily verifies that $\nabla_f^2 L(0, \frac{\mathbf 1}{\| \mathbf 1 \|_p})$ does not exist. 
For $p\geq 2$, we compute
\begin{equation}
    \frac{1}{p(p-1)} \partial_{f_u} \partial_{f_v} L(\lambda, f) = 
        \begin{cases} 
        \sum_{w\in \DV: uw \in \DE} | f_u - f_w|^{p-2} - \lambda |f_u|^{p-2}, &\qquad u=v\\
        -|f_u-f_v|^{p-2}, &\qquad uv\in \DE \\
        0, &\qquad \text{otherwise.}
        \end{cases}
\end{equation}
For $p=2$, the Hessian $\nabla_{f}^2(0, \frac{\mathbf 1}{\| \mathbf 1 \|_p})$ coincides up to constant with the standard Laplacian and the dimension of its kernel is 1 since $\DG$ is connected. However, for $p>2$ we have
\begin{equation}
\nabla_f^2 L\left(0, \frac{\mathbf 1}{\| \mathbf 1 \|_p}\right) = \mathbf 0 \in \mathbb{R}^{|\DV|\times|\DV|}
\end{equation}
and condition \eqref{eq:eig_continuation_cond} is obviously violated (on a non-trivial graph).
\end{example}

\subsection{A case study: triangle graph}
\label{sec:triangle}

In this section we study the $p=4$ Laplacian on the triangle graph, Figure \ref{fig:triangle}(left). Using the symmetry of the graph, we reduce the eigenvalue problem to a single equation which we analyze numerically and, in the special case $\omega\equiv1$, analytically. Then we describe the associated cut problem and identify the eigenvalues of the triangle as extrema of the eigenvalue curves of the cut graph.  In the special case $\omega\equiv1$, we classify all non-regular eigenpairs.  Numerically, we observe that regularity is not a necessary condition for the extremal properties of the eigenvalues.

Consider the triangle graph $\DG =\{ (u, v, w), (uv, vw, uw) \}$, Figure~\ref{fig:triangle},
with $\kappa \equiv 0, \rho \equiv 1$, $\sigma \equiv 1$ and $p=4$. 

\begin{figure}[ht]
\centering
\begin{minipage}{.4\textwidth}
\hfill 
\includegraphics{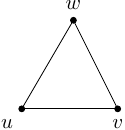}
\end{minipage} \hfill \begin{minipage}{.4\textwidth}\includegraphics{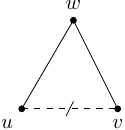} \end{minipage}
\caption{We consider the triangle graph and introduce a cut to reduce the problem to one on a path graph.}
\label{fig:triangle}
\end{figure}

In this case we can reduce the eigenvalue problem to a single equation (plus two explicit eigenvalues).  Writing out the eigenvalue problem explicitly,
\begin{align}\label{eq:eigvalueptri1}
        \omega_{uv} \left (f_u- f_v \right )^3 + \omega_{uw} \left ( f_u - f_w\right )^3 &= \lambda f_u^3, \\ \label{eq:eigvalueptri2}
        \omega_{uv} \left (f_v - f_u \right )^3 + \omega_{vw} \left ( f_v - f_w \right )^3 &= \lambda f_v^3, \\ \label{eq:eigvalueptri3}
        \omega_{uw} \left ( f_w - f_u \right )^3 + \omega_{vw} \left ( f_w - f_v \right )^3 &=\lambda f_w^3,
\end{align}
and adding the three equations, we infer
\begin{equation}
    0 = \lambda (f_u^3 + f_v^3 + f_w^3).
\end{equation}
Then either $\lambda =0$ or 
\begin{equation}
  \label{eq:sum0}
    f_u^3 + f_v^3 + f_w^3 = 0.
\end{equation}
In the former case, $(1, 1, 1)$ is the only eigenvector, as seen in Example \ref{ex:nonregularsimple}.
In the latter case, by scalar invariance of the eigenvalue problem, we need to consider two cases, $f_u=0$ and $f_u=1$.

Under the assumption $f_u=1$, we denote $f_w =x$ and infer
\begin{equation}
    f_v = - (1+ x^3 )^{1/3}.
\end{equation}
From \eqref{eq:eigvalueptri1} we infer
\begin{equation}
  \lambda = \omega_{uv} \left (1+ (1+ x^3)^{1/3} \right )^3 
  + \omega_{uw} \left ( 1 - x\right )^3,
\end{equation}
and plugging this into \eqref{eq:eigvalueptri2} conclude that $x$ must satisfy 
\begin{equation}\label{eq:eigvalueptrisol}
    \omega_{uw} (x-1)^3(1+x^3) + \omega_{vw} \left(x+(1+x^3)^{1/3}\right)^3 - \omega_{uv} x^3\left(1+(1+ x^3)^{1/3}\right)^3 = 0.
\end{equation}
See Figure~\ref{fig:findroot} for the corresponding graph of the expression in \eqref{eq:eigvalueptrisol} with its corresponding roots. 

\begin{figure}[ht]
    \centering
    \includegraphics[scale =0.25]{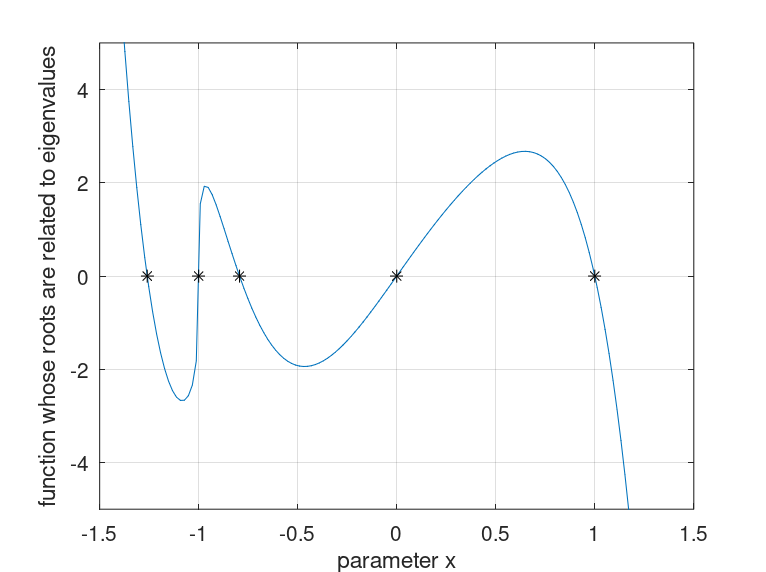}
    \includegraphics[scale=0.28]{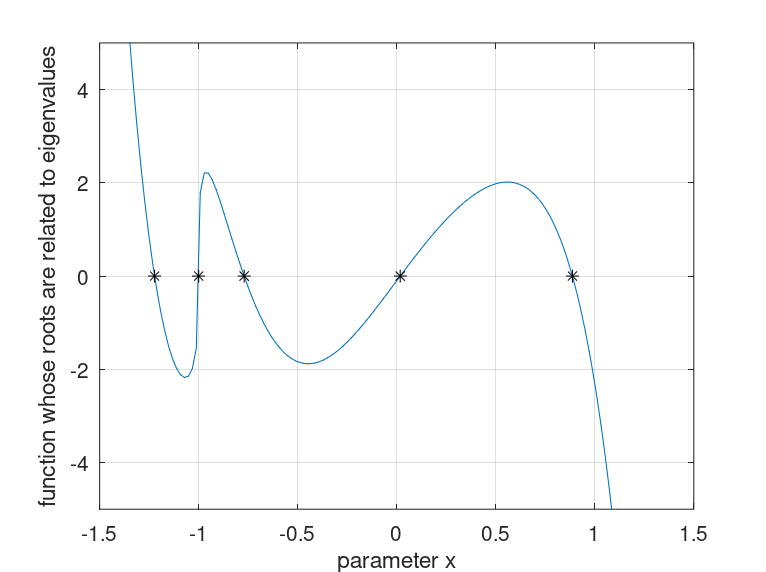}
    \caption{Triangle eigenvalues can be identified as the roots of \eqref{eq:eigvalueptrisol}. (Left: $\omega \equiv 1$, Right: $\omega_{uv} = 1, \omega_{uw}=1.1, \omega_{vw}=0.9$)}
    \label{fig:findroot}
\end{figure}

Finally, if $f_u=0$, using scalar invariance and \eqref{eq:sum0} again, we get $f_v=1$ and $f_w=-1$.  This is a valid eigenvector if and only if $\omega_{uw} = \omega_{uv}$, with the eigenvalue
\begin{equation}
    \label{eq:eig_f0}
    \lambda = \omega_{uv} + 8\omega_{vw},
    \qquad
    (\text{if }\omega_{uw} = \omega_{uv}).
\end{equation}
To connect it with equation \eqref{eq:eigvalueptrisol}, we note that when $\omega_{uw} \to \omega_{uv}$, dominant balance reveals a solution $x\to \pm\infty$ of \eqref{eq:eigvalueptrisol}.  It corresponds to a bounded $\lambda$ which converges to \eqref{eq:eig_f0}.

Furthermore, in the symmetric case $\omega \equiv 1$, we can identify the roots with specific values that we know correspond to eigenpairs, see Table~\ref{tab:triangle_eig_all}.

\setlength{\tabcolsep}{10pt}
\begin{table}[ht]
  \centering
  \begin{tabular}{c|c|c}
    Eigenvalue & Eigenvector & $x$ value in \eqref{eq:eigvalueptrisol}\\[5pt]
    \hline
    $\lambda=0$ & $(1, 1, 1)$ & n/a \\ \hline
    \multirow{3}{*}{$\lambda=1+2^3 = 9$} 
    & $(0, 1, -1)$ & n/a \\
    & $(1, -1, 0)$ & $x=-1$ \\
    & $(-1, 0, 1)$ & $x=0$ \\
    \hline 
    \multirow{3}{*}{$\lambda = \left(1+2^{\frac13}\right)^3\approx11.542$}
    & $\left(1, 1, -2^{\frac13}\right)$ & $x=1$ \\
    & $\left(1, -2^{\frac13}, 1\right)$ 
    & $x=-2^{\frac13} \approx -1.2599$ \\
    & $\left(-2^{\frac13}, 1, 1\right)$ 
    & $x=-2^{-\frac13}\approx -0.7937$ \\    
  \end{tabular}
  \caption{Description of the spectrum of the triangle graph in the case $\omega \equiv 1$.}
  \label{tab:triangle_eig_all}
\end{table}

We now introduce a cut parameter by considering the system of equations 
\begin{align}\label{eq:eigvalueptricut1}
    \omega_{uv} (1- \alpha)^3 f_u^3 + \omega_{uw} (f_u - f_w)^3 = \lambda f_u^3\\ \label{eq:eigvalueptricut2}
    \omega_{uv} (1- \frac{1}{\alpha})^3 f_v^3 + \omega_{vw} (f_v - f_w)^3 = \lambda f_v^3\\ \label{eq:eigvalueptricut3}
    \omega_{uw} (f_w - f_u)^3 + \omega_{vw} (f_w - f_v)^3 = \lambda f_w^3.
\end{align}
By scalar invariance, we may assume the values at the leaf nodes are $f_u =1$ and $f_v =x$, getting from \eqref{eq:eigvalueptricut1} and \eqref{eq:eigvalueptricut2}
\begin{equation}
  \label{eq:fw}
    f_w = 1- \left ( \frac{\lambda - \omega_{uv} (1- \alpha)^3}{\omega_{uw}} \right )^{1/3}
    = x\left ( 1- \left ( \frac{\lambda - \omega_{uv}(1- \frac{1}{\alpha})^3}{\omega_{vw}} \right ) \right )^{1/3}
\end{equation}
Solving for $x$ and substituting into \eqref{eq:eigvalueptricut3}, we obtain the secular equation
\begin{equation}\label{eq:seculartriangle}
\begin{aligned}
    &(\lambda - \omega_{uv} (1-\alpha)^3) \left ( 1- \left ( \frac{\lambda - \omega_{uv}(1- \frac{1}{\alpha})^3}{\omega_{vw}} \right )^{1/3}\right )^3\\
    &\quad + \lambda \left ( 1- \left ( \frac{\lambda - \omega_{uv} (1- \alpha)^3}{\omega_{uw}} \right )^{1/3}\right )^3 \left ( 1- \left ( \frac{\lambda - \omega_{uv}(1- \frac{1}{\alpha})^3}{\omega_{vw}} \right )^{1/3} \right )^3\\
    &\quad + \left (\lambda - \omega_{uv} \left (1-\frac{1}\alpha\right )^3\right ) \left(1- \left (\frac{\lambda- \omega_{uv} (1-\alpha)^3}{\omega_{uv}}\right)^{1/3}\right )^3 =0
\end{aligned}
\end{equation}
The numerically computed eigenvalues $\lambda(\alpha)$ are shown in Figure~\ref{fig:cutsystem}.

\begin{figure}[ht]
    \centering
    \includegraphics[scale=0.27]{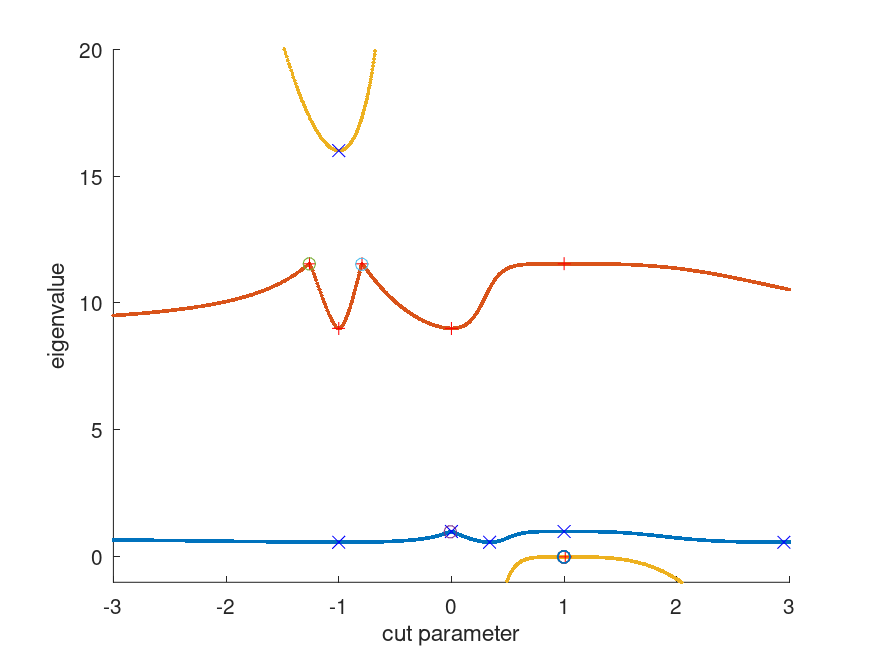}
    \includegraphics[scale=0.25]{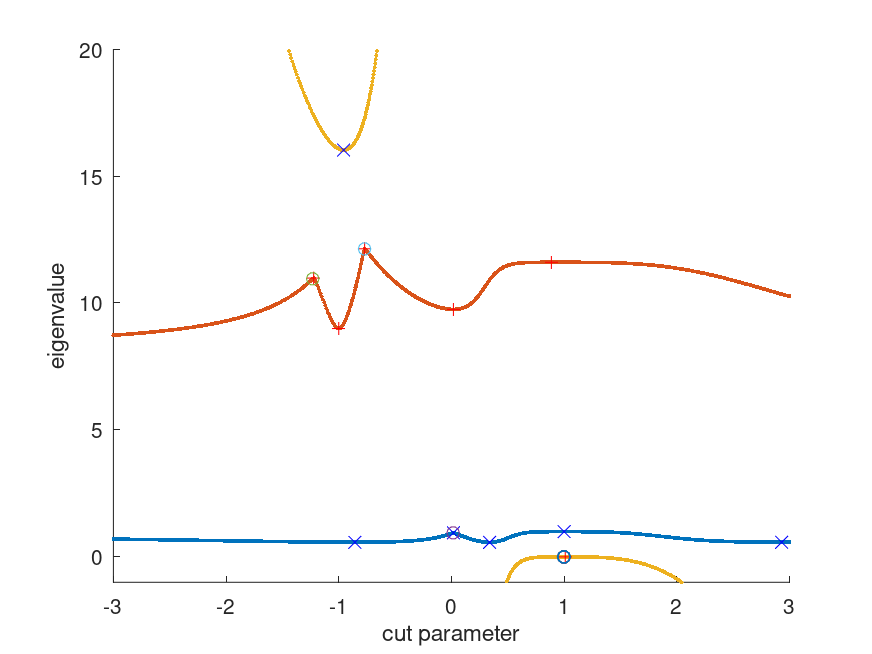}
    \caption{The eigenvalues of the cut system \eqref{eq:eigvalueptricut1}-\eqref{eq:eigvalueptricut3} as functions of $\alpha$. The extrema of the curve can be identified as eigenvalues of the signed $p$-Laplacian of the triangle graph: $\sigma = +1$ (marked {\color{red} $+$}) and $\sigma = -1$ (marked {\color{blue} $\times$}).  Left: $\omega \equiv 1$. Right: $\omega_{uv}=1, \omega_{uw} =1.1, \omega_{vw} =0.9$. The nonregular eigenpairs are marked $\circ$.}
    \label{fig:cutsystem}
\end{figure}

The following theorem characterizes all the regular eigenpairs of $\hHpa^{\hat \sigma, \omega}$ (its proof appears in Appendix \ref{sec:rank_proof}).

\begin{theorem}\label{thm:rank}
Let $(\lambda, f)$ be an eigenpair of $\hHpa^{\hat \sigma, \omega}$ with $p=4$, $\alpha\neq0$, $\sigma\equiv 1$ and $\omega \equiv 1$, then $(\lambda, f)$ is regular if and only if $f_u \neq f_w$ and $f_v \neq f_w$. 

Explicitly,  $(\lambda, f)$ is not a regular eigenpair for $\hHpa^{\hat \sigma, \omega}$ if and only if $\lambda=0$ or $\lambda = \left(1+2^{\frac13}\right)^3\approx11.542$, see Table~\ref{tab:nonreg}.
    \begin{table}[H]
    \centering
    \begin{tabular}{c|c|c}
    Eigenvalue & Eigenvector ($c\in \mathbb R\setminus \{0\}$) & Cut parameter $\alpha$\\[5pt]
    \hline
    $\lambda=0$ & $f=c(1, 1, 1)$ & $\alpha =1$ \\
    \hline 
    \multirow{3}{*}{$\lambda = \left(1+2^{\frac13}\right)^3\approx11.542$}
    & $f=c\left(1, -2^{\frac13}, 1\right)$ 
    & $\alpha=-2^{\frac13} \approx -1.2599$ \\
    & $f=c\left(-2^{\frac13}, 1, 1\right)$ 
    & $\alpha =-2^{-\frac13}\approx -0.7937$ \\    
    \end{tabular}
    \caption{Eigenpairs of $\hHpa$ that are not regular in the case $\omega \equiv 1$.}
    \label{tab:nonreg}
    \end{table}
\end{theorem}

In order to better understand what happens at the identified
nonregular eigenpairs of $\hHpa^{\hat \sigma, \omega}$ we can compute
the derivative of the eigenvalues along the curve corresponding to the
eigenvalues of the $p$-Laplacian via
Theorem~\ref{thm:derivationformula} (see \eqref{eq:derivvary}),
\begin{equation}\label{eq:eigderivative}
\begin{aligned}
    \frac{\mathrm d\lambda}{\mathrm d\alpha} 
    &= 3 \omega_{uv} |1-\alpha|^{2} \frac{\frac{|f_v|^4}{ |\alpha|^4}- |f_u|^4}{\sum_{v\in \DV} \rho(v) |f_v|^4},
\end{aligned}
\end{equation}
with the results plotted in Figure~\ref{fig:eigenvalue_derivative}. The figure illustrates that the first derivatives of the eigenvalue curves are continuous. However, at $\alpha \in \{-2^{-1/3}, -2^{1/3}\}$ the derivative of the eigenvalue curve containing the nonregular eigenpairs at $\alpha \in \{-2^{-1/3}, -2^{1/3}\}$ admit vertical asymptotes.  Remarkably, these asymptotes do not necessarily coincide with the parameter at which an extremum is attained in the nonsymmetric case. In particular, the corresponding eigenpairs become regular. Likewise in the eigenvector entries in Figure~\ref{fig:eigenvector} we observe the vertical asymptote for the eigenfunction curve associated to the nonregular eigenvalue $\lambda =(1+2^{1/3})^3$ at $\alpha \in \{-2^{-1/3}, -2^{1/3}\}$ in the symmetric case $\omega \equiv 1$.

\begin{figure}[ht]
    \centering    
    \includegraphics[scale=0.25]{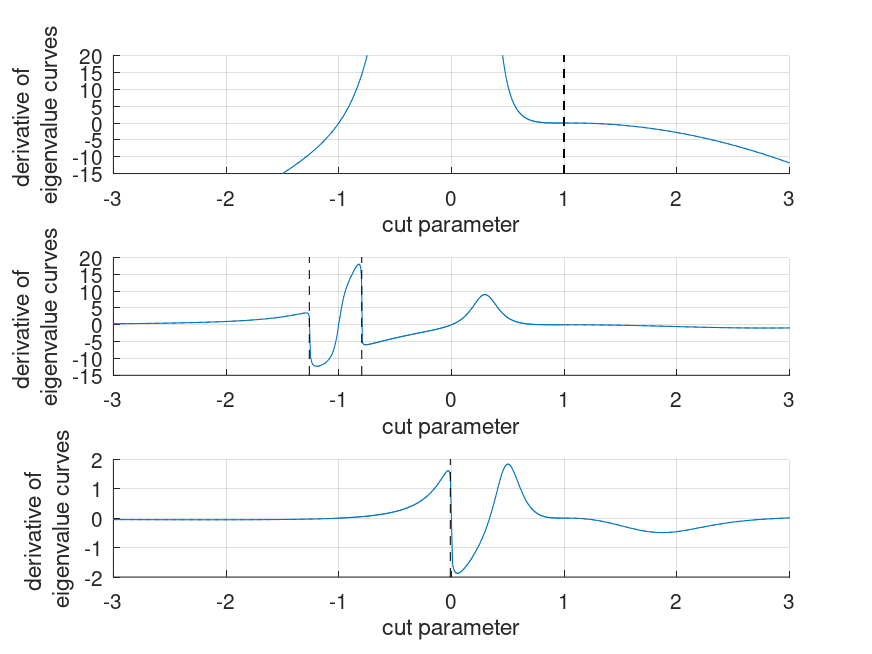}
    \includegraphics[scale=0.25]{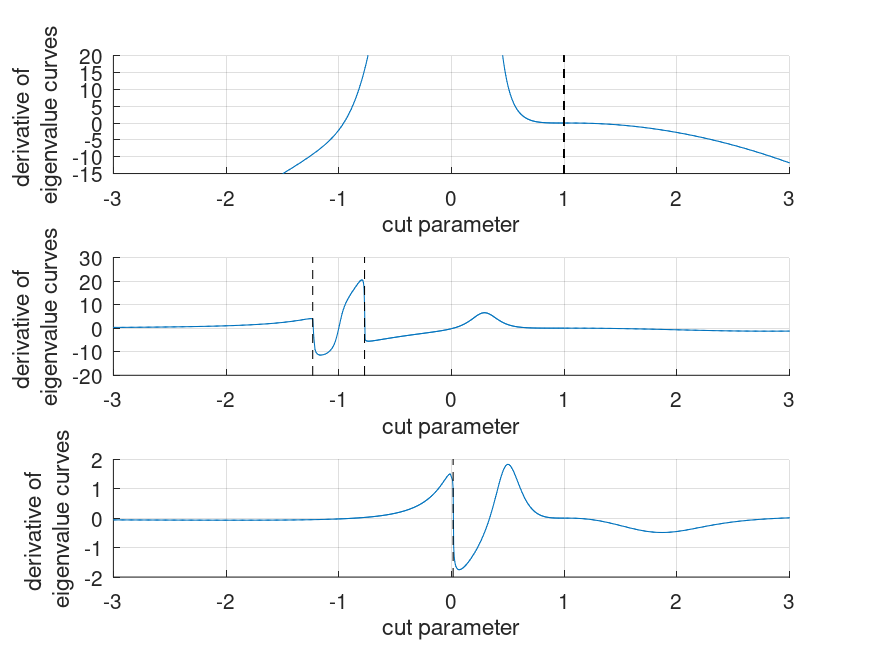}
    \caption{Derivative of the eigenvalues from Figure~\ref{fig:cutsystem}, computed via \eqref{eq:eigderivative}.  The top plot show the derivatives of both the upper and the lower eigenvalue curves from Figure~\ref{fig:cutsystem}.} 
    \label{fig:eigenvalue_derivative}
\end{figure}

\begin{figure}[ht]
    \centering\includegraphics[scale=0.25]{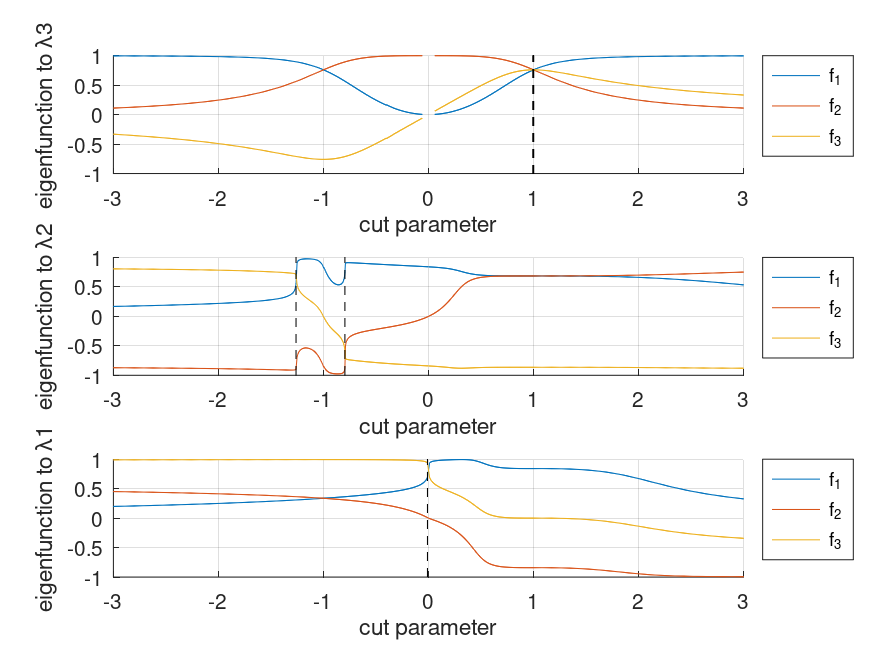}
    \includegraphics[scale=0.25]{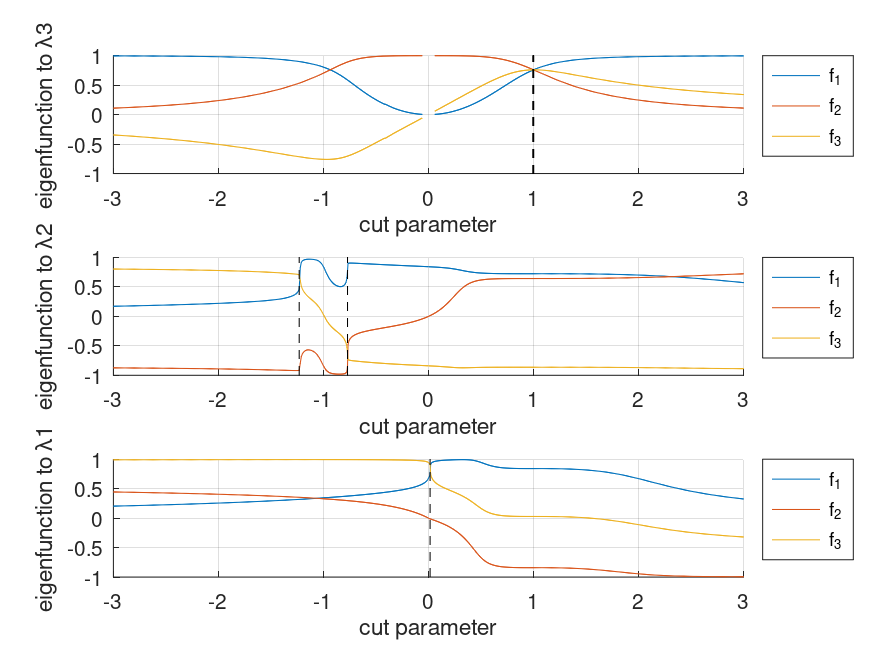}
    \caption{Eigenvectors corresponding to the curves in Figure~\ref{fig:cutsystem}, computed via \eqref{eq:fw} and normalization}
    \label{fig:eigenvector}
\end{figure}

Interestingly, one can see that the nonregular eigenpairs of $\hHpa$ in this examples are also the eigenvalues of $\Hp$ despite their nonregularity, and they can still be identified as extremal points on the eigenvalue curves (see Figure~\ref{fig:cutsystem}).  The behavior at $\alpha \to 0$ and $\alpha \to \infty$, which are excluded from our analysis, also seems quite meaningful, generating extrema near $\alpha=0$ in Figure~\ref{fig:cutsystem} (precisely at $\alpha=0$ in the case $\omega \equiv 1$), which correspond to eigenvalues of the original graph (see Table~\ref{tab:triangle_eig_all} in the case $\omega \equiv 1$).  It is reasonable to conjecture that the value at the $\alpha=0$ can be gleaned from a Dirichlet-type problem: $f_v$ must be set to zero and equation \eqref{eq:eigvalueptricut2} must be removed from consideration.  Similarly, at $\alpha = \pm\infty$ $f_u$ is set to 0 and equation \eqref{eq:eigvalueptricut1} is dropped. In the case $\omega \equiv 1$ these should refer to the missing eigenvalues to $\lambda= 9$ in Table~\ref{tab:triangle_eig_all}.

\subsection{A case study: triangle with a pendant}
\label{sec:attached}

In this section, we consider another case study by adding a pendant edge to a triangle graph. Unlike the previous example, here it is not clear how to compute the eigenvalues of $\Delta_p$ and we apply our cut problem to determine its eigenvalues. In this case, there are two possible ways to introduce cuts, i.e.\ one into a path graph, see Figure~\ref{fig:triangle_attached}(middle), and one into a star graph, see Figure~\ref{fig:triangle_attached}(right). It should be mentioned that for both graphs,  there exist $|\DV|$ eigenvalues of the $p$-Laplacian, see \cite{HeTu18} for the path graphs and \cite{Deidda23, DPT23} for general tree graphs.  Furthermore, the eigenvalues of a path graph are geometrically simple, so the only crossing may occur at $\alpha = 0, \pm\infty$.

Let us now be more precise. In the following we attach a pendant to the triangle graph (see Figure~\ref{fig:triangle_attached}) and consider 
$$\DG = \big\{ \{1,2,3,4\}, \{12, 13, 23, 34\}\big\}.$$

\begin{figure}
\centering
\includegraphics{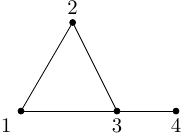}
\qquad
\includegraphics{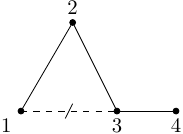}
\qquad
\includegraphics{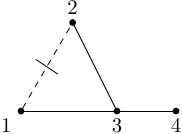}
\caption{(Left) Triangle graph with a pendant edge.  (Middle and Right) two possible cuts considered in this section.}
\label{fig:triangle_attached}
\end{figure}

Consider the signed eigenvalue problem
\begin{equation}\label{eq:tadpole}
 \begin{aligned}
 \omega_{12} (f_{1}- \sigma f_{3})^3 + \omega_{13} (f_{1}-f_{2})^3 &=\lambda  f_{1}^3\\
 \omega_{23} (f_{2}- f_{3})^3 + \omega_{12} (f_{2}- f_{1})^3   &= \lambda f_{2}^3\\
 \omega_{13} (f_{3}- \sigma f_{1})^3 + \omega_{23} (f_{3}- f_{2})^3 + \omega_{24} (f_{3}-f_{4})^3 &= \lambda f_{3}^3\\
 \omega_{34} (f_{4}-f_{3})^3 &= \lambda f_{4}^3
 \end{aligned}
\end{equation}
with $\sigma=\pm 1$. 

If we introduce a cut parameter (first according to Figure~\ref{fig:triangle_attached}, then the problem becomes 
\begin{equation}
    \begin{aligned}
        \omega_{12} (1 - \alpha )^3 f_{1}^3 + \omega_{13} (f_{1}- f_{2})^3 &= \lambda f_{1}^3\\
         \omega_{23} (f_{2}- f_{3})^3 + \omega_{24} (f_{2}-f_{1})^3  &= \lambda f_{2}^3\\
 \omega_{13} \left (1- \frac{1}{\alpha}\right )^3 f_3^3 + \omega_{23} (f_{3}- f_{2})^3 + \omega_{34} (f_3 - f_4)^3 &= \lambda f_{3}^3\\
 \omega_{34} (f_{4}-f_{3})^3 &= \lambda f_{4}^3,
    \end{aligned}
\end{equation}
which can be reduced to a single equation.  Indeed, we set
$f_{4}=1$ and we infer
\begin{equation}
\begin{aligned}
f_{3} &= 1-\left (\frac{\lambda}{\omega_{34}}\right )^{1/3}\\
f_{2} & = f_3 - \left (\frac{(\lambda - \omega_{13}(1-\frac{1}{\alpha})^3) f_3^3 - \omega_{34} (f_3 - 1)^3 }{\omega_{23}}\right )^{1/3}\\
f_{1} &= f_2^3- \left ( \frac{\lambda f_2^3 - \omega_{23} (f_2 - f_3)^3 }{\omega_{12}}\right )^{1/3}
\end{aligned}
\end{equation}
and one can reduce $$\omega_{12} (1 - \alpha )^3 f_{1}^3 + \omega_{13} (f_{1}- f_{2})^3 = \lambda f_{1}^3$$ into a single equation which we solved for $\lambda$ as a function of $\alpha$, see Figure~\ref{fig:crossingsymmetric}(left).  We observe that the criticlal points yield the eigenvalues of the original graph (or its signed version).

\begin{figure}
    \centering
    \includegraphics[scale=.25]{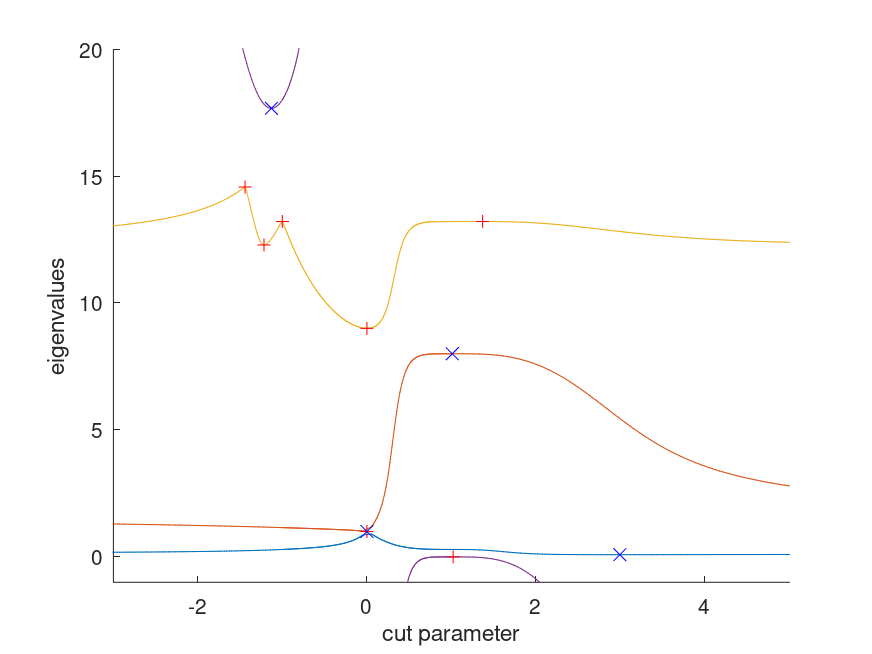}
    \includegraphics[scale=.25]{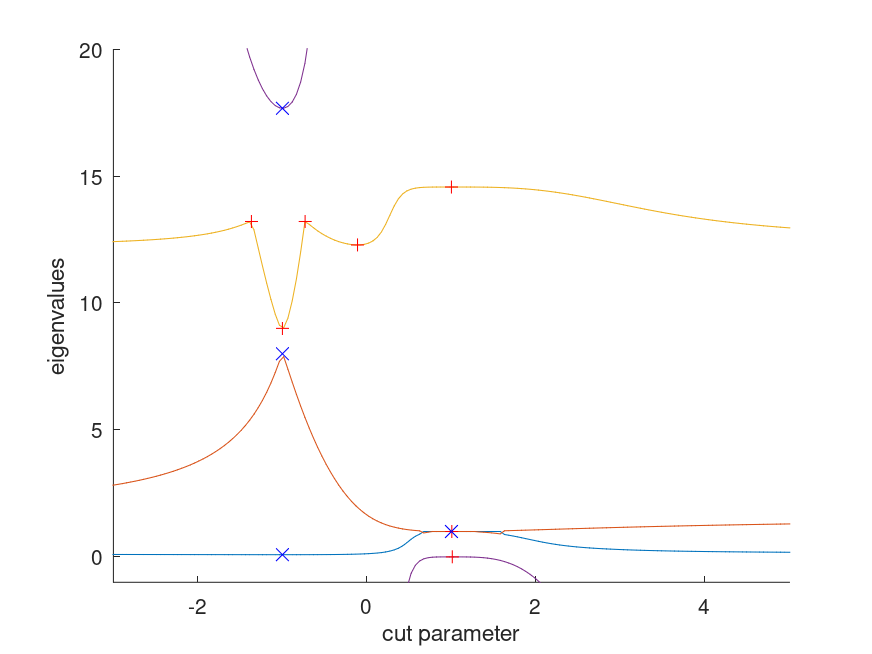}
    \caption{Eigenvalues of the signed $p$-Laplacian with $\omega \equiv 1$ are shown as the extrema of the eigenvalue curves of cut graphs from Figure~\ref{fig:triangle_attached}.  
    Red {\color{red} $+$} mark the eigenvalues of the $\sigma = +1$ Laplacian;  blue {\color{blue} $\times$} mark the eigenvalues of the $\sigma = -1$ Laplacian.  The eigenvalue curves on the left are from a path graph and can only intersect at $\alpha =0, \pm\infty$; the eigenvalue curves of a star can have other intersections.}
    \label{fig:crossingsymmetric}
\end{figure}

Cutting the graph according to Figure~\ref{fig:triangle_attached}(right), we obtain another perturbed problem,
\begin{equation}
\label{eq:second_cut_problem}
\begin{aligned}
    \omega_{12} (1-\alpha)^3 f_1^3 + (f_1-f_3)^3  &=\lambda f_1^3\\
    \omega_{12} (1-\tfrac{1}{\alpha})^3 f_2^3 +\omega_{23} (f_2 - f_3)^3  &= \lambda f_2^3\\
    \omega_{13} (f_3 - f_1)^3 +\omega_{23} (f_3-f_2)^3 + \omega_{34} (f_3- f_4)^3 &= \lambda f_3^3\\
    \omega_{34} (f_4 - f_3)^3 &= \lambda f_4^3.
\end{aligned}
\end{equation}
It can be reduced to a single equation by setting $f_3=1$ and deriving
\begin{equation}
\begin{aligned}
 f_1 &= \left ( 1- \left ( \frac{\lambda - \omega_{12}(1-\alpha)^3}{\omega_{13}} \right )^{1/3} \right )^{-1}\\
 f_2 &= \left ( 1- \left ( \frac{\lambda - \omega_{12} (1- \tfrac{1}{\alpha})^3}{\omega_{23}} \right )^{1/3} \right )^{-1}\\
 f_4 &= \left ( 1- \left ( \frac{\lambda}{\omega_{34}} \right )^{1/3} \right ).
\end{aligned}
\end{equation}
Then one uses the third equation in \eqref{eq:second_cut_problem} to find the eigenvalues for every $\alpha$.  They are shown in Figure~\ref{fig:triangle_attached}(right) and their critical \emph{values} are the same as in Figure~\ref{fig:triangle_attached}(left).

As explained in Section~\ref{sec:triangle}, the eigenvalues at $\alpha=0,\pm\infty$ can be obtained from a simpler Dirichlet-type problem.  We conjecture that they are once differentiable (when simple).  Finally, we observe that when two eigenvalue curves intersect, the intersection point is an eigenvalue of both versions of the signed Laplacian.  It is unknown whether this is true in general.

\section{Final remarks and further directions}
\label{sec:conclusion}

In the course of the paper we have made several empirical observations that, for now, remain open questions.  Namely, the (simple) eigenvalues of tree graphs appear to be differentiable even when the eigenpair is not regular (the singularities appear in second derivative of the eigenvalue or the first derivative of the eigenvector).  Furthermore, even when the eigenvalues are multiple, they appear to be resolvable as $C^1$ curves, but probably only in 1 parameter, in analogy to the Rellich--Kato theorem.
Finally, the eigenvalues appear to be $C^1$ even at the points where our parameters make the potential infinite at some vertices.

In addition to above, it is natural to investigate the type of the critical points: whether they are minima or maxima or, when there are more cut parameters, what is their Morse index.  In the linear case, $p=2$, the Morse index is determined by the sign changes in the eigenvector \cite{BBRS12, BRS12}, which can also be interpreted as the spectral shift \cite{BerKuc_jst22}.  If a similar result is established here, one can potentially combine the spectral shift information with Morse theory to obtain an apriori bound on the number of the eigenvalues of a graph with cycles.  This would extend the results of \cite{DPT23} and would further facilitate the numerical search for the eigenvalues.

\appendix
\section{Proof of Theorem \ref{thm:rank}}\label{sec:rank_proof}

\begin{theorem}\label{thm:criterionregularitytri}
Let $(\lambda, f)$ be an eigenpair of $\hHpa^{\hat \sigma, \omega}$ with $p=4$, $\alpha\neq0$, $\sigma\equiv 1$, then $(\lambda, f)$ is regular if and only if $f_u \neq f_w$ and $f_v \neq f_w$. 
\end{theorem}

\begin{proof}
The Hessian can be directly computed to be 
\begin{equation}
    \frac{1}{4}\nabla_{f}^2 L(\alpha, \lambda, f) = \begin{pmatrix}  F_1 & 0 & - 3\omega_{uw}(f_u- f_w)^2\\  0 &  F_2 & - 3\omega_{vw}(f_v-f_w)^2 \\  -3 \omega_{uw}(f_w- f_u)^2 & -3 \omega_{vw}(f_w-f_v)^2 & F_3 \end{pmatrix},
\end{equation}
where
\begin{equation}
    \begin{aligned}
        F_1&= 3 \omega_{uv} (1-\alpha)^3 f_u^2 + 3 \omega_{uw} (f_u-f_w)^2 - 3\lambda f_u^2 \\
        F_2 &= 3 \omega_{uv} (1-\tfrac{1}{\alpha})^3 f_v^2 + 3 \omega_{vw}  (f_v-f_w)^2 - 3\lambda f_v^2 \\
        F_3&= 3 \omega_{uw}(f_w-f_u)^2 + 3 \omega_{vw} (f_w-f_v)^2 - 3 \lambda f_w^2
    \end{aligned}   
\end{equation}
Then if $f_u\neq f_w$ and $f_v \neq f_w$, then $(f_v - f_w)^2 \neq 0 \neq (f_u - f_w)^2$, and the first and third column vectors are linearly independent and we have $\operatorname{rank}(\nabla_f^2 L) \ge 2$. Suppose $f_u = f_w$ or $f_v = f_w$, then $f_w\neq 0$ since by \eqref{eq:eigvalueptricut3} otherwise $f_u = f_v = f_w =0$. Then by either \eqref{eq:eigvalueptricut1} or \eqref{eq:eigvalueptricut2} we can deduce similarly $f_u \neq 0$ and $f_v \neq 0$. 

Hence, due to \eqref{eq:eigvalueptricut1}, \eqref{eq:eigvalueptricut2} and \eqref{eq:eigvalueptricut3} we have
\begin{equation}
    \begin{aligned}
        F_1&= 3 \omega_{uv} (1-\alpha)^3 f_u^2 + 3 \omega_{uw} (f_u-f_w)^2 - 3\lambda f_u^2= 3 \omega_{uw} (f_u - f_w)^2 \frac{f_w}{f_u} \\
        F_2 &= 3 \omega_{uv} (1-\tfrac{1}{\alpha})^3 f_v^2 + 3 \omega_{vw}  (f_v-f_w)^2 - 3\lambda f_v^2 = 3\omega_{vw} (f_v - f_w)^2 \frac{f_w}{f_v} \\
        F_3&= 3 \omega_{uw}(f_w-f_u)^2 + 3 \omega_{vw} (f_w-f_v)^2 - 3 \lambda f_w^2 \\
        &=3 \omega_{uw} (f_w - f_u)^2 \frac{f_u}{f_w} + 3 \omega_{vw} (f_w-f_v)^2 \frac{f_v}{f_w}.
    \end{aligned}   
\end{equation}
and one easily verifies that if $f_u = f_w$ or $f_v = f_w$, then the eigenpair $(\lambda, f)$ is nonregular.  
\end{proof}

\begin{proof}[Proof of Theorem~\ref{thm:rank}]
The Hessian in this case can be directly computed to be
\begin{equation}\label{eq:hessianproof}
    \frac{1}{4}\nabla_{f}^2 L(\alpha, \lambda, f) = \begin{pmatrix}  F_1 & 0 & - 3(f_u- f_w)^2\\  0 &  F_2 & - 3(f_v-f_w)^2 \\  -3 (f_w- f_u)^2 & -3 (f_w-f_v)^2 & F_3 \end{pmatrix},
\end{equation}
where 
\begin{equation}
    \begin{aligned}
        F_1&= 3 (1-\alpha)^3 f_u^2 + 3 (f_u-f_w)^2 - 3\lambda f_u^2\\
        F_2 &= 3 (1-\tfrac{1}{\alpha})^3 f_v^2 + 3  (f_v-f_w)^2 - 3\lambda f_v^2 \\
        F_3&= 3 (f_w-f_u)^2 + 3 (f_w-f_v)^2 - 3 \lambda f_w^2.
    \end{aligned}   
\end{equation}

We already saw in Example \ref{ex:nonregularsimple} that $\lambda = 0$, which corresponds to $\alpha = 1$, is not a regular eigenvalue. Then by Theorem~\ref{thm:criterionregularitytri} any eigenpair is nonregular if and only if $f_u = f_w$ or $f_v=f_w$.

If $f_u = f_v = f_w$, then \eqref{eq:eigvalueptricut3} implies $\lambda =0$ and \eqref{eq:eigvalueptricut1} implies $\alpha =1$. In particular, $F_1 = F_2 = F_3 =0$ and $\nabla_f^2 F =0$ and $(\lambda, f)$ is not regular. 

Suppose $f_u = f_w$ and $f_v \neq f_w$, then \eqref{eq:eigvalueptricut1} implies
\begin{equation}
    \lambda f_u^3 = (1-\alpha)^3 f_u^3.
\end{equation}
Note that $f_u \neq 0$ since otherwise \eqref{eq:eigvalueptricut3} would imply  $f_v = f_w=f_u=0$. Hence, $f_u = f_w \neq 0$ and
\begin{equation}\label{eq:case1eigvalueptri}
    \lambda =  (1- \alpha)^3.
\end{equation}
Summing \eqref{eq:eigvalueptricut1}, \eqref{eq:eigvalueptricut2}, \eqref{eq:eigvalueptricut3} we have
\begin{align}\label{eq:eigvalueptrisum1}
    \lambda ( f_u^3 + f_v^3 + f_w^3) = (1- \alpha)^3 \left ( f_u^3 - \frac{f_v^3}{\alpha^3}\right ).
\end{align}
Recall that $\alpha \neq 1$.  Without loss of generality $f_u= f_w=1$, then with  \eqref{eq:case1eigvalueptri} and \eqref{eq:eigvalueptrisum1} we have
\begin{equation}
    \left (1 + \frac{1}{\alpha^3}\right ) f_v^3 = -1.
\end{equation}
In particular, $\alpha \neq -1$ and $f_v \neq 0$. Thus,
\begin{equation}\label{eq:case1eigvalueptri2}
    f_v = - \frac{1}{(1+ \frac{1}{\alpha^3})^{1/3}}.
\end{equation}
and with \eqref{eq:eigvalueptricut3} we have
\begin{equation}
    \lambda =  (1-f_v)^3 
\end{equation}
and with \eqref{eq:case1eigvalueptri} we have $f_v = \alpha$ and \eqref{eq:case1eigvalueptri2} leads to
\begin{equation}
\alpha = -\frac{1}{(1+ \frac{1}{\alpha^3})^{1/3}}
\end{equation}
and we infer $\alpha= - 2^{1/3}$. 

Then \eqref{eq:hessianproof} computed at $f_u = f_w=1, f_v = -2^{1/3}$ and $\alpha = -2^{1/3}$ becomes
\begin{equation}
    \nabla^2_f L \big |_{f_u=f_w=1, f_v =-2^{1/3}; \alpha = -2^{1/3}} = \begin{pmatrix} 0 & 0 & 0 \\ 0 & -3 \;2^{-1/3} (1+ 2^{1/3})^2 & - 3 (1 + 2^{1/3})^2 \\ 0 & - 3 (1+ 2^{1/3})^2 & - 3\; 2^{1/3} (1+ 2^{1/3})^2 \end{pmatrix}
\end{equation}
and $\Rank(\nabla_f^2 L|_{f_u = f_w =1, f_v = - 2^{-1/3}}) = 1$. 

The case $f_v = f_w$ and $f_u \neq f_w$ is identical via interchanging $v$ and $u$ and $\alpha$ and $\alpha^{-1}$. Hence, $\Rank(\nabla_f F) =1$ independent of the considered cases and $(\lambda, f)$ is not regular as an eigenpair of $\hHpa^{\hat \sigma, \omega}$. One now easily verifies that the table contains all the cases when the eigenpair $(\lambda, f)$ is not regular and that these indeed are eigenvalues of the original problem.
\end{proof}


\end{document}